\documentclass[12pt,a4paper]{amsart}
\usepackage{times}
\usepackage{amsmath,amssymb,amsthm,enumerate}
\usepackage{datetime,color,graphicx}
\usepackage{hyperref,multicol}
\hypersetup{pdfcreationdate=\pdfdate,pdfmoddate=\pdfdate,colorlinks=true,citecolor=black,linkcolor=black,urlcolor=black}
\usepackage[top=1in, bottom=1in, left=1in, right=1in]{geometry}

\usepackage{titlesec}
\titleformat{\section}{\bfseries}{\thesection.}{1em}{}

\numberwithin{equation}{section}
\newtheorem{theorem}[equation]{Theorem}
\newtheorem*{thmnonum}{Theorem}
\newtheorem*{mainthm}{Main Theorem}
\newtheorem{lemma}[equation]{Lemma}
\newtheorem{cor}[equation]{Corollary}

\theoremstyle{definition}
\newtheorem{definition}[equation]{Definition}

\newcommand{\abs}[1]{\left\lvert#1\right\rvert}

\DeclareMathOperator{\Har}{\operatorname{Har}}

\DeclareMathOperator{\Lip}{\operatorname{Lip}}

\begin{document}

\title{Transversally Lipschitz harmonic functions are Lipschitz}
\author{Sivaguru Ravisankar}
\address{Department of Mathematics, The Ohio State University, Columbus, Ohio 43210}
\email{\href{mailto:sivaguru@math.ohio-state.edu}{\nolinkurl{sivaguru@math.ohio-state.edu}}}
\date{\usdate\today}


\begin{abstract}
Let $\Omega\subset\mathbb{R}^n$ be a bounded domain with $C^\infty$ boundary. We show that a harmonic function in $\Omega$ that is Lipschitz along a family of curves transversal to $b\Omega$ 
is Lipschitz in $\Omega$.  The space of Lipschitz functions we consider is defined using the notion of a majorant which is a certain generalization of the power functions $t^\alpha$, $0<\alpha<1$.	
\end{abstract}

\maketitle

\section{Introduction}
The purpose of this paper is to show that transverse Lipschitz regularity transfers to all directions for a harmonic function on a bounded domain with $C^\infty$ boundary. The space of Lipschitz functions we 
consider is defined using the notion of a majorant (see Definition  \ref{defn:Majorant}). A majorant is a certain generalization of the power functions $t^\alpha$, $0<\alpha<1$. This generalization allows us to 
highlight the key properties of the power function $t^\alpha$ that enter the analysis here.

Let $\Omega\subset\mathbb{R}^n$ be a bounded domain with $C^\infty$ boundary.
\begin{definition} Let $B$ be a majorant. A function $f$ defined in $\Omega$ is called {\it Lipschitz-$B$} if $\exists\, C_f>0$ such that
\[\abs{f(x)-f(y)} \le\, C_f\cdot B(\abs{x-y}), \ \forall\, x,y\in\Omega .\]
Let $\Lambda_B(\Omega)$ denote the set of Lipschitz-$B$ functions on $\Omega$.
\end{definition}
The classical Lipschitz (or H\"{o}lder) spaces correspond to the majorants $B(t)=t^{\alpha}$, $0<\alpha<1$. We 
call Lipschitz-$t^\alpha$ functions as Lipschitz-$\alpha$ functions and we denote $\Lambda_{t^\alpha}\left(\Omega\right)$ by $\Lip_{\alpha}\left(\Omega\right)$. Readers not interested in this generalization may replace every occurrence of the majorant (or regular majorant) $B(t)$ with the function $t^\alpha$, $0<\alpha<1$, and our result is interesting even in this special case.

We use a family of curves transversal to $b\Omega$ (see Definition \ref{defn:FamTransCurves}) to measure transverse regularity. Let $\Gamma$ be such a family. A function $f$ defined in $\Omega$ is said to be transversally 
Lipschitz-$B$ with respect to $\Gamma$ if the restriction of $f$ to each curve of $\Gamma$ in $\Omega$ is (uniformly) Lipschitz-$B$. Our main theorem is as follows. Let $\Har(\Omega)$ denote the set of 
harmonic functions defined in $\Omega$.

\begin{mainthm} Let $\Omega \subset \mathbb{R}^n$ be a bounded domain with $C^\infty$ boundary and let $\Gamma$ be a family of curves transversal to $b\Omega$. Let $u\in\Har(\Omega)$ and 
$B$ be a regular majorant. If $u$ is transversally Lipschitz-$B$ with respect to $\Gamma$, then $u\in\Lambda_B(\Omega)$.
\end{mainthm}
In particular, for $B(t)=t^\alpha, 0<\alpha <1$, we have the following corollary.
\begin{cor} Let $\Omega \subset \mathbb{R}^n$ be a bounded domain with $C^\infty$ boundary and let $\Gamma$ be a family of curves transversal to $b\Omega$. Let $u\in\Har(\Omega)$ and 
$0<\alpha<1$. If $u$ is transversally Lipschitz-$\alpha$ with respect to $\Gamma$, then $u\in\Lip_\alpha(\Omega)$.
\end{cor}
We outline the proof of the Main Theorem for the special case $\Omega=\mathbb{B}$. The proof in this special case captures all the key ideas behind the result. The general case is handled by attaching 
($\mathbb{R}^n$-) sectors of balls to $b\Omega$ and then using the result for $\mathbb{B}$. 

Let $u$ be harmonic in $\mathbb{B}$ and uniformly Lipschitz-$B$ along $\Gamma$, a family of curves transversal to $b\mathbb{B}$. To show the conclusion $u\in\Lambda_B(\mathbb{B})$, it suffices by the 
Hardy-Littlewood Theorem (Theorem \ref{thm:H-L}) to show that
\[\abs{\nabla u(x)} \,\lesssim\, \frac{B(\delta(x))}{\delta(x)}, \text{ for } x\in U\cap\mathbb{B},\]
where $U$ is a neighbourhood of $b\mathbb{B}$ and $\delta(x)$ is the Euclidean distance of $x$ to $b\mathbb{B}$.

We use a scaling 
argument via $u_\lambda(x)=u(\lambda x)$, $1/2 < \lambda < 1$. We exploit the fact that $u_\lambda$ is harmonic in $\mathbb{B}$, $u_\lambda\in C^\infty\left(\overline{\mathbb{B}}\right)$, and $u_\lambda$ is 
Lipschitz-$B$ along $\Gamma_\lambda$, a suitable perturbation of $\Gamma$. Let $M$ be the unit vector field given by differentiation along curves of $\Gamma_\lambda$. We show that $Mu_\lambda$ grows
no faster than the rate prescribed by the Hardy-Littlewood theorem, modulo an error term involving a small constant times $\nabla u_\lambda$. We accomplish this by using a constant 
coefficient approximation of the vector field $M$, call it $M_0$,  and estimating the second derivatives of $u_\lambda$ by its first derivatives in its Taylor expansion along the curves of 
$\Gamma_\lambda$. We then show that, for a constant coefficient vector field $N_0$ that is orthonormal to $M_0$, the rate of growth of $N_0u_\lambda$ is similar to that of $M_0u_\lambda$. Combining these 
estimates, we show that $\nabla u_\lambda$ has a rate of growth no worse than that prescribed by the Hardy-Littlewood theorem modulo an error term involving a small constant times $\nabla u_\lambda$. We 
absorb the small constant times $\nabla u_\lambda$ into the $\nabla u_\lambda$ term to show that $\nabla u_\lambda$ grows no worse than the rate prescribed by the Hardy-Littlewood theorem. Since the 
constants in our estimates are independent of $\lambda$, we let $\lambda \to 1$ to finish the proof.

Our result generalizes a result of Pavlovi\'{c} \cite{Pav07-LipHarmonic} which states that the Lipschitz behaviour in the radial direction of a harmonic function in $\mathbb{B}$ transfers to all directions, where $\mathbb{B}$ is the unit ball in $\mathbb{R}^n$. 
\begin{thmnonum}[Pavlovi\'{c}, 2007] Let $u\in\Har(\mathbb{B}) \cap C(\overline{\mathbb{B}})$ and $B$ be a regular majorant. If $\exists\, C>0$  such that 
\[\abs{u(\zeta)-u(r\zeta)} \le C\cdot B(1-r), \text{ for } \zeta\in b\mathbb{B}\text{, } 0<r<1,\]
then $u\in\Lambda_B(\mathbb{B})$.
\end{thmnonum}
His proof hinges on $r\frac{\partial u}{\partial r}$ and $r^2\frac{\partial^2 u}{\partial r^2}$ being harmonic in $\mathbb{B}$ for $u\in\Har(\mathbb{B})$. Also, the rate of growth of these radial derivatives encode the 
radial Lipschitz behaviour of $u$. In contrast, the estimates used in proving our result are significantly more involved since we do not have a differential operator that both preserves harmonic functions and also 
encodes their transverse Lipschitz behaviour along a family of curves transversal to the boundary.

The following result of D\'{e}traz \cite{Det81} is in the same spirit as ours in the setting of weighted $L^p$ regularity. 
\begin{thmnonum}[D\'{e}traz, 1981] Let $u\in\Har(\Omega)$ and $L$ be a continuous unit vector field in a neighbourhood of $b\Omega$ and transverse to $b\Omega$. 
Then, 
\[Lu\in L^p_a(\Omega) \implies \nabla u \in L^p_a(\Omega),\] 
for $p>0$ and $a>-1$. Here, 
\[L^p_a(\Omega) = \left\{ f \text{ measurable on } \Omega\, :\, \int\limits_\Omega\, \abs{f(x)}^p\,\delta(x)^a\, dx < \infty\right\}\]
where $\delta(x)$ is the Euclidean distance of $x$ to $b\Omega$.
\end{thmnonum}

On a related front, Dyakonov \cite{Dya97} showed that the Lipschitz-$B$ norm of $f$ and $\abs{f}$ are equivalent for a holomorphic function $f\in\Lambda_B(\overline{\mathbf{D}})$, where $\mathbf{D}$ is the unit 
disk in the complex plane. Pavlovi\'{c} \cite{Pav99} has given a simpler and more elegant proof of this. Pavlovi\'{c} \cite{Pav07-LipHarmonic} has also considered the equivalence between several Lipschitz-$B$ and radial Lipschitz-$B$ norms of $f$ and $\abs{f}$, on $\mathbb{B}$ and $b\mathbb{B}$, where $f$ is a real valued harmonic function in $\mathbb{B}$.

This paper is organized as follows. In Section 2, we recall the definition of a majorant and its important properties, and state the Hardy-Littlewood theorem. We present the key tools involved in the proof of the 
Main Theorem in Lemmas \ref{lemma:HarEst}, \ref{lemma:TransDistEst}, and \ref{lemma:EstOutsideCpt} in Section 3. Section 4 is devoted to the proof of the Main Theorem, first for the special case 
$\Omega=\mathbb{B}$ in Theorem \ref{thm:TransLipBall}, and then for the general case.

We also fix the following notation. $A\subset\subset B$ will mean that $A\subset B$ and has compact closure in $B$. Also, we use $a\,\lesssim\,b$ or $b\,\gtrsim\,a$ to mean $a \le 
Cb$ for some constant $C>0$ which is independent of certain parameters. It will be mentioned, or clear from the context, what these parameters are. We use $a \approx b$ to mean $a\,\lesssim\,b$ and $b\,
\lesssim\,a$. We call a function or the boundary of a domain smooth if it is $C^\infty$ smooth.

\section{Lipschitz Functions: Majorants and Hardy-Littlewood Theorem}

Majorants and their regularity appear in the work of Dyakonov \cite{Dya97} and go back at least to the work of Havin \cite{Hav71} and Zygmund \cite{Zyg59}, if not any earlier.
\begin{definition} \label{defn:Majorant} A continuous function $B : [0,\infty) \rightarrow [0,\infty)$ is called a \textit{majorant} if
\[B(0)=0,\  B \text{ is non-decreasing, and }\, \frac{B(t)}{t}\, \text{ is non-increasing}.\]
\end{definition}

Clearly, for $0<\alpha\le 1$, $t^\alpha$ is a majorant. The functions $-t^{\alpha}\ln t$, for $0<\alpha \le 1$, and $1/\left(\ln t\right)^2$ (for $t$ near $0$) are majorants. For a majorant $B$, the condition on a function being Lipschitz-$B$ is a local one. So, we only focus on the behaviour of $B$ near 0. This suggests that the more the 
majorant $B$ behaves like $t^{\alpha}$ near 0 the more we can expect Lipschitz-$B$ functions to behave like Lipschitz-$\alpha$ functions. The following integral estimate on $B$ ensures that.

\begin{definition}\label{defn:RegMaj}A majorant funtion $B$ is called \textit{regular} if $\exists\, C>0,\ \forall\, \delta >0$ sufficiently small,
\begin{equation}\label{eq:RegMaj}
\int\limits_0^{\delta}\, \frac{B(t)}{t}\, dt + \delta\,\int\limits_{\delta}^{\infty}\, \frac{B(t)}{t^2}\, dt\, \le\, C\cdot B(\delta).
\end{equation}
\end{definition}

The majorants $t^\alpha$ and $-t^{\alpha}\ln t$ are regular for $0<\alpha <1$, whereas the majorants $t$, $-t\ln t$, and $1/(\ln t)^2$ are not regular. The inequality \eqref{eq:RegMaj} can be naturally broken up into 
two inequalities. A majorant satisfying each of these inequalities can be characterized by related functions being almost increasing or almost decreasing. For more on this see \cite[Proposition 1]{Pav07-LipHarmonic} or \cite[Section 2A]{Rav11}.

We now recall a theorem of Hardy and Littlewood which gives a sufficient condition for a function to be Lipschitz-$B$, where $B$ is a regular majorant, in terms of the rate of growth of its derivative. We sketch a 
proof for the readers convenience.

\begin{theorem}[Hardy-Littlewood]\label{thm:H-L}Let $\Omega \subset\subset \mathbb{R}^n$ have smooth boundary and let $B$ be a regular majorant. Let $U$ be a neighbourhood of $b\Omega$. If $f \in C^{1}\left(\Omega\right) \cap L^{\infty}\left(\Omega\right)$ satisfies
\[\abs{\nabla f(x)} \, \lesssim\, \frac{B\left(\delta(x)\right)}{\delta(x)}\ , \quad x \in U\cap\Omega ,\]
where $\delta(x)$ is the Euclidean distance of $x$ to $b\Omega$, then $f\in\Lambda_B\left(\Omega\right)$.
\end{theorem}

\begin{proof}Notice that it suffices to show that $f$ is Lipschitz-$B$ near $b\Omega$. Fix $0 < \delta_0 < 1$ so that $V := \left\{ x\in\Omega \, :\, \delta(x) < 3\delta_0\right\}\subset U\cap\Omega$. Let $T,S\in V$ 
such that $\abs{T-S}<\delta_0$. The estimate on $\nabla f$ is in terms of the distance to the boundary. To show that $f$ is in $\Lambda_{B}$ we need to compare the function values at $T$ and $S$ in $V$. We 
achieve this by pushing these points inside $\Omega$ by a fixed $\epsilon$ so that we can use the estimate on $\nabla f$. We then choose $\epsilon$ effectively to achieve the result.
\begin{figure}[htp]
\centering
\includegraphics[width=0.5\textwidth]{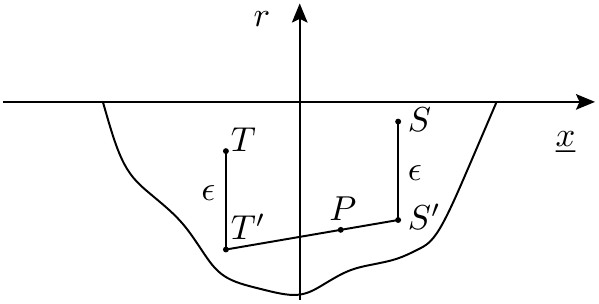}
\caption{Box Argument - Hardy-Littlewood}
\end{figure}

Let $r$ be the signed distance to $b\Omega$, i.e.,
\[r(x) = \begin{cases}
-\delta(x), & \text{if } x\in\Omega, \text{ and}\\
\delta(x), & \text{if } x\notin\Omega.
\end{cases}\]
$r$ is smooth near $b\Omega$ and a defining function for $\Omega$, i.e., $\Omega = \{r<0\}$, $b\Omega = \{ r=0 \}$, and $\abs{\nabla r}\ne 0$ on $b\Omega$\footnote{For more on the distance to the boundary 
function, see Gilbarg-Trudinger \cite[pp.~354-357]{Gil-Tru83} and Herbig-McNeal \cite{Her-McN10}.}. Decrease $\delta_0$, if necessary, so that $r$ is smooth in $V$ and we may assume, without loss of any 
generality, that $\partial r/\partial x_n \neq 0$ near $T$ and $S$. So, we consider $r$ to be a coordinate in the normal direction on $V$, i.e., $(x_1,\ldots,x_{n-1},r)$ are coordinates on $V$. Let $T=(t_1,
\ldots,t_{n-1},t_n),\, S=(s_1,\ldots, s_{n-1}, s_n) \in V$. For $0<\epsilon \le \delta_0$, let $T'=(t_1,\ldots,t_{n-1},t_n-\epsilon)$ and  $S'=(s_1,\ldots,s_{n-1},{s_n-\epsilon})$. Since $r$ is a coordinate in the normal 
direction, we know that $T', S' \in \Omega$. Also, for any $P$ in the line $L'$, in the $(x_1,\ldots,x_{n-1},r)$ coordinate system, joining $T'$ and $S'$,  $\delta(P) > \epsilon$.
\begin{align*}
\abs{f(T')-f(S')} &\le \abs{\nabla f(P)}\abs{T'-S'}  \, \lesssim\,  \frac{B\left(\delta(P)\right)}{\delta(P)}\cdot \abs{T-S}\quad (\text{for some }P \in L')\\
&\, \lesssim\, \frac{B(\epsilon)}{\epsilon} \cdot \abs{T-S}\quad (\text{since } B(x)/x \text{ is non-increasing}).
\end{align*}
Choosing $\epsilon=\abs{T-S}$, we get $\abs{f(T')-f(S')} \,\lesssim\, B(\abs{T-S})$. We now estimate $\abs{f(T)-f(T')}$.
\begin{align*}
\abs{f(T)-f(T')} &= \abs{\ \int\limits_0^{\epsilon} \frac{\partial f}{\partial r}(t_1,\ldots,t_{n-1},t_n-x)\, dx\ } \le \int\limits_0^{\epsilon} \frac{B(-t_n+x)}{-t_n+x}\, dx \\
&\le \int\limits_0^{\epsilon} \frac{B(x)}{x}\, dx\,\, \lesssim\, B(\abs{T-S}) \quad\text{(since $B$ is regular)}.
\end{align*}
Similary, one estimates $\abs{f(S)-f(S')}$.
\end{proof}
For harmonic functions, the converse of the above theorem is also true. Hence, the condition on the rate of growth of a harmonic function's derivative characterizes the function being Lipschitz.
\begin{lemma} Let $\Omega\subset\subset\mathbb{R}^n$ have smooth boundary and let $u\in\Har(\Omega)$. If $u\in\Lambda_B(\Omega)$, then
\[\abs{\nabla u(x)}\, \lesssim\, \frac{B(\delta(x))}{\delta(x)},\ \forall x\in\Omega .\]
\end{lemma}
\begin{proof} Fix $x_0\in\Omega$. Let $\epsilon=\delta(x_0)/2$. So, $B(x_0,\epsilon) \subset\subset \Omega$. Now, by the Poisson integral formula, for $x\in B(x_0,\epsilon)$
\begin{align*}
\nabla u(x) &= \frac{1}{\omega_{n-1}\epsilon} \int\limits_{\abs{\xi}=\epsilon} u\left(x_0+\xi\right) \nabla_x\left(\frac{\epsilon^2 - \abs{x-x_0}^2}{\abs{x-x_0-\xi}^n}\right)\, d\sigma(\xi)\\
&= \frac{1}{\omega_{n-1}\epsilon} \int\limits_{\abs{\xi}=\epsilon} \left(u\left(x_0+\xi\right)-u(x_0)\right) \nabla_x\left(\frac{\epsilon^2 - \abs{x-x_0}^2}{\abs{x-x_0-\xi}^n}\right)\, d\sigma(\xi) .
\end{align*}
Calculating $\nabla_x$ inside the integral, setting $x=x_0$, and estimating we get
\[\abs{\nabla u(x_0)} \le \frac{n}{\epsilon}\cdot\sup\limits_{\abs{\xi}=\epsilon} \abs{u(x_0+\xi)-u(x_0)}\, \lesssim\, \frac{B(\epsilon)}{\epsilon} \,\lesssim\, \frac{B(\delta(x_0))}{\delta(x_0)}.\]
\end{proof}

\section{Key Tools Used in the Proof of the Main Theorem}

There are three key estimates we use in the proof of the Main Theorem. We state and prove them in the following lemmas. The first one allows us to estimate the values of the derivative of a harmonic function 
on a ball by the values of the harmonic function on a larger concentric ball.

\begin{lemma}\label{lemma:HarEst} Let $u$ be a harmonic function on $\mathbb{B}$ and let $0<r<R<1$. Then,
    \[\sup\limits_{\abs{x}=r} \abs{\nabla u(x)} \le \frac{n}{R-r} \cdot \sup\limits_{\abs{x}=R} \abs{u(x)}.\]
\end{lemma}
\begin{proof} Fix $x_0\in \mathbb{B}$ such that $\abs{x_0}=r$. Let $\epsilon=R-r$. By Poisson integral formula, for $x\in B(x_0;\epsilon)$, 
\[\nabla u(x) = \frac{1}{\omega_{n-1} \epsilon} \int\limits_{\abs{\xi-x_0}=\epsilon} u(\xi) \cdot \nabla_x \left(\frac{\epsilon^2 - \abs{x-x_0}^2}{\abs{\xi-x}^n} \right) d\sigma(\xi).\]
Calculating $\nabla_x$ inside the integral and setting $x=x_0$, we get
\[\left.\nabla_x \left(\frac{\epsilon^2 - \abs{x-x_0}^2}{\abs{\xi-x}^n} \right)\right|_{x=x_0} = -\frac{n}{\epsilon^{n}}\cdot (\xi-x_0)\]
and hence
\[ \abs{\nabla u(x_0)} \le \frac{1}{\omega_{n-1}\epsilon}\cdot\sup\limits_{\abs{\xi-x_0}=\epsilon}\abs{u(\xi)} \cdot \frac{n}{\epsilon^{n-1}} \cdot \omega_{n-1}\epsilon^{n-1} \le \frac{n}{\epsilon}\cdot\sup\limits_{\abs{\xi}\le R}\abs{u(\xi)}.\]
The result follows by using the maximum principle and then taking supremum over $\abs{x_0}=r$.
\end{proof}

We will use the above lemma on a ball contained in $\Omega$ whose centre is obtained by moving a point in $\Omega$ near $b\Omega$ along a direction transversal to $b\Omega$. The following 
lemma helps us estimate the radius of such a ball. Recall that $\delta(x)$ is the Euclidean distance of $x$ to $b\Omega$. For $p\in b\Omega$, let $\nu_p$ denote the outward unit normal to $b\Omega$ at 
$p$.

\begin{lemma}\label{lemma:TransDistEst}
Let $\vec{v}$ be a unit vector that is transverse to $b\Omega$ at $p$. i.e., $\exists\, c>0$ such that $ \vec{v}\cdot\nu_p \le -c$.
Then, for $0<a<1$, $\exists\, S_a>0$ so that
		\[acs \le \delta\left(p+s\vec{v}\right) \le s, \text{ for } 0 < s \le S_a.\]
\end{lemma}

\begin{proof}
Let $\left\{e_1,\ldots, e_n\right\}$ denote the standard basis for $\mathbb{R}^n$. Fix $p\in b\Omega$. By a rotation of $\Omega$ we may assume that $\nu_p  = e_n$. Let $\vec{v} = 
\langle v_1, \ldots, v_n\rangle$ be as in the statement. Then, $v_n\le -c < 0$. Let $r$ be the smooth defining function for $\Omega$ given by the signed distance to the boundary. We know that $\nabla r(p) = \nu_p = e_n$ (see \cite[Corollary 5.3]{Her-McN10}).
 
It is clear that $\delta(p+s\vec{v})\le s$. Now, let us show the other inequality. For small $s>0$, we have $p+s\vec{v} \in \Omega$. Let $C$ be the maximum of all the second derivatives of $r$ in a 
neighbourhood of $b\Omega$. Then, by Taylor's Theorem $\exists\, w \in \Omega$ near $b\Omega$, such that $p-w$ is parallel to $\vec{v}$ and
\[r(p+s\vec{v}) =  r(p) +s\sum_{j=1}^n\, \frac{\partial r}{\partial e_j}(p)v_j + \frac{s^2}{2}\sum_{j,k=1}^n\, \frac{\partial^2 r}{\partial e_j \partial e_k}(w)v_jv_k.\]
Since $\abs{r(x)}=\delta(x)$ and $p\in b\Omega$, $r(p)=0$. Also, since $\nabla r(p)=e_n$, we have 
$\nabla r(p) \cdot \vec{v} = v_n$. Hence,
\[\delta(p+s\vec{v})=\abs{r(p+s\vec{v})} \geq s\abs{v_n} - Cs^2 \geq s(c-Cs) = cs \left(1-\frac{C}{c}s\right).\]
Choose $S_a>0$ so that $1-\frac{C}{c}S_a \geq a$.
\end{proof}

To show the conclusion of the Main Theorem, it suffices by the Hardy-Littlewood Theorem to show
\[\abs{\nabla u(x)} \,\lesssim\, \frac{B(\delta(x))}{\delta(x)}, \text{ for } x\in U\cap\Omega,\]
where $U$ is a neighbourhood of $b\Omega$, or equivalently,
\[\sup\limits_{U\cap\Omega}\, \abs{\nabla u(x)}\frac{\delta(x)}{B(\delta(x))}  < \infty .\]
In the proof of Theorem \ref{thm:TransLipBall}, the main theorem for $\mathbb{B}$, we consider three such neighbourhoods and we will be comparing the above suprema on those neighbourhoods. The following lemma allows us to do 
that. This is a consequence of the maximum principle for sub-harmonic functions.

\begin{lemma}\label{lemma:EstOutsideCpt}
Let $B$ be a majorant. For $0<\delta_0<\delta_1$, let 
\[U_0 := \left\{ x\in\mathbb{R}^n \, :\, \delta(x) < \delta_0\right\} \quad\text{and}\quad U_1 := \left\{ x\in\mathbb{R}^n\, :\, \delta(x) < \delta_1\right\}.\]
If $u$ is a harmonic function in $\Omega$ satisfying
\[\sup\limits_{U_0\cap\Omega }\ \abs{\nabla u(x)}\frac{\delta(x)}{B(\delta(x))} = A_0 < \infty,\]
then
\[\sup\limits_{U_1\cap\Omega}\ \abs{\nabla u(x)}\frac{\delta(x)}{B(\delta(x))} \le \frac{\delta_1}{\delta_0}\cdot A_0\]
\end{lemma}

\begin{proof} Let $w\in \left(U_1\setminus U_0\right)\cap\Omega $. By the maximum principle for the sub-harmonic function $\abs{\nabla u}^2$,
$\exists\, x_w\in \Omega$ with $\delta(x_w) = \delta_0$ such that $\abs{\nabla u(w)} \le \abs{\nabla u(x_w)}$. By continuity, we have
\[\abs{\nabla u(x_w)}\frac{\delta(x_w)}{B(\delta(x_w))} \le A_0.\]
Hence,
\begin{align*}
\abs{\nabla u(w)}\frac{\delta(w)}{B(\delta(w))}&\le \abs{\nabla u(x_w)}\cdot \frac{\delta_1}{B(\delta_1)} \quad(\text{since } t/B(t) \text{ is non-decreasing})\\
&=\frac{\delta_1}{B(\delta_1)} \cdot \abs{\nabla u(x_w)} \cdot \frac{\delta(x_w)}{B(\delta(x_w))}\cdot\frac{B(\delta(x_w))}{\delta(x_w)}\\
&\le \frac{\delta_1}{B(\delta_1)}\cdot A_0\cdot\frac{B(\delta_0)}{\delta_0} \le \frac{\delta_1}{\delta_0}\cdot A_0 \quad(\text{since } B \text{ is non-decreasing}).
\end{align*}
\end{proof}	

\section{Main Theorem}

In this section we prove the Main Theorem which states that a transversally Lipschitz harmonic function is Lipschitz. We first consider the special case corresponding to the unit ball $\mathbb{B}$ in Theorem 
\ref{thm:TransLipBall}. The proof in this special case captures all the key ideas behind the result. We then prove the Main Theorem by attaching ($\mathbb{R}^n$-) sectors of balls to $b\Omega$ and 
then using the result for $\mathbb{B}$. Let us begin by defining the necessary notions. Let $\Omega \subset\subset \mathbb{R}^n$ have smooth boundary.

\begin{definition}\label{defn:FamTransCurves}
Let $U$ be a neighbourhood of $b\Omega$ and $\Gamma : b\Omega \times (-a,a) \to U$ be a $C^2$ map (for some $a>0$). For $p\in b\Omega$ and $t\in (-a,a)$, let $\gamma_p(t) := \Gamma(p,t)$.
$\Gamma$ is called a \textit{family of curves transversal to $\mathit{b\Omega}$} if the following hold;
\begin{enumerate}[(a)]
\item $\gamma_p(0) = p, \text{ for } p\in b\Omega$, and
\item $\exists\, c>0$ such that
\[\gamma_p'(t) \cdot \nu_p \le -c < 0, \text{ for }  p\in b\Omega \text{ and } t\in (-a,a).\]
\end{enumerate}
\end{definition}
Transversality, for us, means $\gamma_p'(t)\cdot\nu_p \ne 0$ for $p\in b\Omega$ and $t\in (-a,a)$. Using compactness of $b\Omega$, the continuity of 
$\gamma_p'$, and by restricting $t$ to a closed sub-interval around $0$, we get that this inner product is uniformly bounded away from $0$. By making the negative choice for sign we get condition (b) above.

We decrease $a$ and correspondingly shrink $U$, if necessary, so that $\Gamma$ is a $C^1$ bijection near $b\Omega$, and for $p\in b\Omega$, 
\[\gamma_p((0,a)) \subset U\cap \Omega \quad\text{and}\quad \gamma_p((-a,0)) \subset U\cap \Omega^c.\]
\begin{definition} Let $B$ be a majorant and $\Gamma$ be a family of curves transversal to $b\Omega$. A function $f$ defined on $\Omega$ is said to be \textit{transversally Lipschitz-$B$} along $\Gamma$ if there exists $C_f > 0$ such that for all $p\in b\Omega$ and $s,t>0$ and sufficiently small,
\[\abs{f(\gamma_p(s))-f(\gamma_p(t))} \le C_f\cdot B(\abs{s-t}).\]
\end{definition}

\begin{theorem}\label{thm:TransLipBall}
Let $\Gamma$ be a family of curves transversal to $b\mathbb{B}$. If $u$ is harmonic in $\mathbb{B}$ and Lipschitz-$B$ along $\Gamma$, then $u\in\Lambda_B(\mathbb{B})$.
\end{theorem}

\begin{proof}
We re-parametrize $\Gamma$ by the arc-length starting at $b\Omega$ to get $\abs{\frac{d}{dt}\left(\gamma_p(t)\right)}=1$. This does not affect the transversality of 
$\Gamma$ or $u$ being Lipschitz-$B$ with respect to $\Gamma$. So, $\Gamma : b\mathbb{B}\times (-a,a) \to V$ is a $C^1$ bijection for some $a>0$ and $V$ a neighbourhood of $b\mathbb{B}$. It suffices, by 
Hardy-Littlewood Theorem, to show
\begin{equation}\label{eq:GradEst}\sup\limits_{U\cap\mathbb{B}}\, \abs{\nabla u(x)}\frac{\delta(x)}{B(\delta(x))} \le C < \infty\end{equation}
for some neighbourhood $U$ of $b\mathbb{B}$. Let
\[C_u^U := \sup\limits_{U\cap\mathbb{B}}\, \abs{\nabla u(x)}\frac{\delta(x)}{B(\delta(x))}.\]
Notice that if $u$ is in $C^1\left(\overline{\mathbb{B}}\right)$, then condition \eqref{eq:GradEst} is automatically satisfied with $C$ depending on $u$. This is our starting point. For $\frac{1}{2} < \lambda < 1$, 
define $u_{\lambda}(x) := u(\lambda x)$. Note that $u_{\lambda}\in C^{\infty}\left(\overline{\mathbb{B}}\right)$ and harmonic in $\mathbb{B}$. Since $t/B(t)$ is non-decreasing,
\[C_{u,\lambda}^U := \sup\limits_{U\cap\mathbb{B}}\, \abs{\nabla u_\lambda (x)}\frac{\delta(x)}{B(\delta(x))} \le ||\abs{\nabla u_\lambda}||_{\infty}\cdot\frac{\operatorname{diam} \mathbb{B}}{B\left(\operatorname{diam} \mathbb{B}\right)} < \infty.\] 
We will show that $u_\lambda$ is Lipschitz-$B$ along $\Gamma_\lambda$, a family of curves transversal to $b\mathbb{B}$, which is related to $\Gamma$. Using this we then show that $C_{u,\lambda}^U$ can 
indeed be dominated by a constant independent of $\lambda$. We conclude that $u$ satisfies \eqref{eq:GradEst} by letting $\lambda \to 1$.

Since $\Gamma$ gives a foliation of $V$ by curves, we get a projection $\pi_\Gamma : V\cap \mathbb{B} \to b\mathbb{B}$ along $\Gamma$, i.e., for $x\in V\cap\mathbb{B}$, $\exists!\, \pi_\Gamma(x)\in 
b\mathbb{B}$ and $0 < T_x < a$, such that $\Gamma(\pi_\Gamma(x), T_x) = x$. For 
simplicity of notation let us drop the subscript $\Gamma$ in $\pi_\Gamma$ and simply call it $\pi$. Define $\Gamma_\lambda$ by, 
\[\Gamma_\lambda(p,t) = \frac{1}{\lambda}\cdot\Gamma\left(\pi(\lambda p), t+T_{\lambda p}\right), \quad p\in b\mathbb{B} \text{ and } \abs{t} \text { small}.\]
\begin{figure}[h]
\begin{center}
\includegraphics[width=0.43\textwidth]{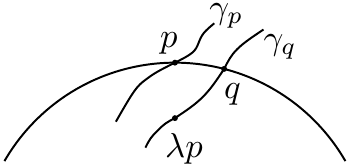}
\caption{Defining $\Gamma_\lambda$}
\label{GammaLambdaDefn}
\end{center}
\end{figure}

We restrict $\lambda$ sufficiently close to 1 to make $\Gamma_\lambda$ well-defined near $b\mathbb{B}$. In all of the analysis that follows we will be working in a small 
neighbourhood of such a boundary point. We will exploit this localization when we generalize the result to a general bounded domain with smooth boundary. First, let us verify that 
$\Gamma_\lambda$ is a family of curves transversal to $b\mathbb{B}$. Let $p$, $\lambda p$ and $q$ be as in Figure \ref{GammaLambdaDefn}. Clearly,
\[\Gamma_\lambda(p,0) = \frac{1}{\lambda}\cdot\Gamma\left(\pi(\lambda p), T_{\lambda p}\right)=\frac{1}{\lambda}\cdot\lambda p=p.\]
Now, let us check transversality. Since, $\gamma_q\left(T_{\lambda p}\right) = \lambda p$, we have
\[\lambda p - q = \gamma_q\left(T_{\lambda p}\right) - \gamma_q(0) = T_{\lambda p}\gamma_q'(t^*), \text{ for some } 0 < t^* < T_{\lambda p}.\]
So, $p=\frac{1}{\lambda}\left(q+T_{\lambda p}\gamma_q'(t^*)\right)$. Hence,
\begin{align*}
\frac{\partial\Gamma_\lambda}{\partial t}(p,t)\cdot \nu_p &= \frac{\partial\Gamma_\lambda}{\partial t}(p,t)\cdot p = \frac{1}{\lambda}\cdot\frac{\partial\Gamma}{\partial t}(q,t+T_{\lambda p})\cdot\frac{1}{\lambda}\left(q+T_{\lambda p}\gamma_q'(t^*)\right)\\
&=\frac{1}{\lambda^2}\left(\gamma_q'(t+T_{\lambda p})\cdot q + T_{\lambda p}\gamma_q'(t+T_{\lambda p})\cdot\gamma_q'(t^*)\right)\\
&\le \frac{1}{\lambda^2}\left(-c+T_{\lambda p}\right).
\end{align*}
As $\lambda \to 1$, $T_{\lambda p} \to 0$. So, choose $\lambda \ge (1/2)$ close enough to $1$ so that $T_{\lambda p} \le (c/2)$ for all $p\in b\mathbb{B}$. Then, we have
\[\frac{\partial\Gamma_\lambda}{\partial t}(p,t)\cdot \nu_p \le -\frac{c}{2}.\]
Without loss of generality let us suppose that the Lipschitz-$B$ constant of $u$ along $\Gamma$ is $1$. Then, we also have
\begin{align*}
\abs{u_\lambda\left(\Gamma_\lambda(p,s)\right)-u_\lambda\left(\Gamma_\lambda(p,t)\right)} &= \abs{u\left(\Gamma\left(\pi(\lambda p), s+T_{\lambda p}\right)\right)-u\left(\Gamma\left(\pi(\lambda p), t+T_{\lambda p}\right)\right)}\\
&= \abs{u\left(\gamma_q(s+T_{\lambda p})\right)-u\left(\gamma_q(t+T_{\lambda p})\right)}\\
&\le B\left(\abs{s-t}\right).
\end{align*}
This shows that $u_\lambda$ is Lipschitz-$B$ along $\Gamma_\lambda$. Also, notice that
\[\frac{\partial\Gamma_\lambda}{\partial t}(p,t) = \frac{1}{\lambda}\cdot\frac{\partial\Gamma}{\partial t}(q,t+T_{\lambda p}),\text{ and hence } \abs{\frac{\partial\Gamma_\lambda}{\partial t}(p,t)} = \frac{1}{\lambda}.\]

Let us denote $u_\lambda$ by $v$ and $\Gamma_\lambda$ by $\widehat{\Gamma}$. Let $\widehat{\gamma}$ denote the curves of $\widehat{\Gamma}$. Let $\widehat{\pi}$ be the projection along the curves of $\widehat{\Gamma}$ and for $x\in V$, let $\widehat{T}_x$ be such 
that $\widehat{\Gamma}\left(\widehat{\pi}(x), \widehat{T}_x\right) = x$. Let $M$ be the unit vector field given by differentiation along curves of $\Gamma_\lambda$, i.e.,
\[Mf(x) = \lambda\nabla f(x) \cdot \left.\frac{\partial}{\partial t}\left(\widehat{\Gamma}\left(\widehat{\pi}(x),t\right)\right)\right|_{t=\widehat{T}_x} = \lambda\nabla f(x) \cdot \widehat{\gamma}_{\widehat{\pi}(x)}'(\widehat{T}_x).\]
Let us now choose two neighbourhoods of $b\mathbb{B}$ to work in. Let 
\[0 < \epsilon < \min\left\{\frac{1}{6},  \frac{c^2}{6400n(n-1)C_2}\right\},\]
where $C_2$ is the constant from the inequality \eqref{eq:RegMaj}. By the uniform continuity of the tangent vectors to the curves of $\Gamma$, $\exists\, \delta_0 > 0$ such that
\[x,y\in V\cap\mathbb{B}, \text{ and }\abs{x-y} < \delta_0 \quad\implies\quad \abs{\widehat{\gamma}_{\widehat{\pi}(x)}'(\widehat{T}_x) - \widehat{\gamma}_{\widehat{\pi}(y)}'(\widehat{T}_{y})} < \epsilon.\]
Decrease $\delta_0$, if necessary, so that $\delta_0 < \epsilon$ and 
$\left\{x\in\mathbb{R}^n \, :\, \delta(x) < 2\delta_0\right\}$ is contained in the neighbourhood corresponding to the bijection given by $\widehat{\Gamma}$, and also contained in the part of the neighbourhood coming from Lemma 
\ref{lemma:TransDistEst} that lies in $\Omega$. Let 
\[U := \left\{ x\in\mathbb{R}^n \, :\, \delta(x) < \delta_0\right\}\quad\text{ and }\quad U' := \left\{ x\in\mathbb{R}^n \, :\, \delta(x) < \frac{\delta_0}{4}\right\}.\]

Fix $x_0\in U'\cap\mathbb{B}$. Now we begin estimating $Mv(x_0)$. Let $p=\widehat{\pi}(x_0)$. For $s>\widehat{T}_{x_0}$, we have
\[v\left(\widehat{\gamma}_{p}(s)\right) - v\left(\widehat{\gamma}_{p}(\widehat{T}_{x_0})\right) = (s-\widehat{T}_{x_0})\cdot \frac{Mv(x_0)}{\lambda} + \frac{(s-\widehat{T}_{x_0})^2}{2}\cdot\frac{d^2}{dt^2}\left(v\left(\widehat{\gamma}_p(t)\right)\right),\]
for some $\widehat{T}_{x_0} < t < s$. Since $v$ is Lipschitz-$B$ along $\widehat{\gamma}_p$, we have
\begin{equation}\label{eq:Dagger}
\abs{Mv(x_0)} \le \frac{B(s-\widehat{T}_{x_0})}{s-\widehat{T}_{x_0}} + \frac{s-\widehat{T}_{x_0}}{2}\cdot\abs{\frac{d^2}{dt^2}\left(v\left(\widehat{\gamma}_p(t)\right)\right)}.
\end{equation}
In what follows, we will estimate any dependence on $\lambda$ using $1/2 < \lambda < 1$ so that the inequalities we obtain are independent of $\lambda$. On many occasions this step may not be shown 
explicitly. 

Let $\vec{M_0}= \lambda\widehat{\gamma}_p'(\widehat{T}_{x_0})$ and $M_0 f(x) = \lambda\nabla f(x) \cdot \vec{M_0}$. Let 
\begin{equation}\label{eq:sChoice}
s= \widehat{T}_{x_0}+\dfrac{\kappa\delta(x_0)}{4}\min\left\{\frac{c^2}{1024n} ,\frac{c}{4C_1}\right\},
\end{equation}
where $0 < \kappa < 1$ is to be chosen later and $C_1$ is the maximum of the length of the second derivatives of the curves of $\widehat{\Gamma}$. So, 
\[\abs{\widehat{\gamma}_p(t) - x_0} = \abs{\widehat{\gamma}_p(t) - \widehat{\gamma}_p(\widehat{T}_{x_0})} \le 2\abs{t-\widehat{T}_{x_0}} \le \frac{\delta(x_0)}{2} < \frac{\delta_0}{8}.\]
 Hence, $\widehat{\gamma}_p(t)\in U\cap\mathbb{B}$, and $\abs{Mv(\widehat{\gamma}_p(t)) - M_0v(\widehat{\gamma}_p(t))} < \epsilon\abs{\nabla v(\widehat{\gamma}_p(t))}$. Using this, we compute the last term in \eqref{eq:Dagger} to get
\begin{multline}\label{eq:DaggerDagger}
\abs{\frac{d^2}{dt^2}\left(v\left(\widehat{\gamma}_p(t)\right)\right)} \le \frac{1}{\lambda^2}\abs{\nabla(M_0 v)(\widehat{\gamma}_p(t))} + C_1\abs{\nabla v (\widehat{\gamma}_p(t))} \\
+ 2\epsilon\max\left\{\abs{v_{xx}}, \abs{v_{xy}}, \abs{v_{yy}}\right\}(\widehat{\gamma}_p(t)).
\end{multline}
Now, we want to use Lemma \ref{lemma:HarEst} to estimate the first and last term. To do this we need to estimate $\displaystyle \delta\left(\widehat{\gamma}_p(t)\right)$. We now show that $\displaystyle \delta\left(\widehat{\gamma}_p(t)\right) \ge c\delta(x_0)/8$;
\[\widehat{\gamma}_p(t) - p = \widehat{\gamma}_p(t) - \widehat{\gamma_p}(0) = t\,\widehat{\gamma}_p'(t^*) \quad\text{for some } 0< t^* < t.\]
Hence, by Lemma \ref{lemma:TransDistEst} with $a=1/2$, we have
\[\delta(\widehat{\gamma}_p(t)) = \delta\left(p+t\, \widehat{\gamma}_p'(t^*)\right) \ge \frac{c}{4}\cdot t\cdot \abs{\widehat{\gamma}_p'(t^*)} = \frac{ct}{4\lambda} \ge \frac{ct}{4}\ge \frac{c\widehat{T}_{x_0}}{4}.\]
Similarly,
\[\delta(x_0) = \delta(\widehat{\gamma}_p(\widehat{T}_{x_0})) = \delta\left(p+\widehat{T}_{x_0}\cdot \widehat{\gamma}_p'(t^{**})\right) \le \widehat{T}_{x_0}\cdot\abs{\widehat{\gamma}_p'(t^{**})} = \frac{\widehat{T}_{x_0}}{\lambda} \le 2\widehat{T}_{x_0}.\]
Combining this with the previous inequality, we get
\[\delta(\widehat{\gamma}_p(t)) \ge \frac{c\delta(x_0)}{8}.\]
Notice in \eqref{eq:DaggerDagger} that $v_x$ and $v_y$ are harmonic in $\mathbb{B}$. Since $M_0$ is a constant coefficient vector field, $M_0 v$ is harmonic too. By Lemma \ref{lemma:HarEst} we have the following;
\[\abs{\nabla(M_0 v)(\widehat{\gamma}_p(t))} \,\le\, \frac{16n}{c\delta(x_0)}\sup\left\{ \abs{M_0v(y)}\, :\, \abs{y-\widehat{\gamma}_p(t)} = \frac{c\delta(x_0)}{16}\right\},\]
and
\[\max\left\{\abs{v_{xx}}, \abs{v_{xy}}, \abs{v_{yy}}\right\}(\widehat{\gamma}_p(t)) \le \frac{16n}{c\delta(x_0)}\sup\left\{ \abs{\nabla v(y)} : \abs{y-\widehat{\gamma}_p(t)} = \frac{c\delta(x_0)}{16}\right\}.\]
For $\abs{y-\widehat{\gamma}_p(t)} = c\delta(x_0)/16$, we have $\delta(y) \ge c\delta(x_0)/16$ and hence
\begin{align*}
\abs{M_0v(y)} &= \abs{M_0v(y)} \cdot \frac{\delta(y)}{B(\delta(y))} \cdot \frac{B(\delta(y))}{\delta(y)} \le \abs{M_0v(y)} \cdot\frac{\delta(y)}{B(\delta(y))}\cdot \frac{B\left(\frac{c\delta(x_0)}{16}\right)}{\left(\frac{c\delta(x_0)}{16}\right)}\\
& \le \frac{16}{c}\cdot\frac{B(\delta(x_0))}{\delta(x_0)}\cdot\abs{M_0v(y)} \cdot \frac{\delta(y)}{B(\delta(y))} \le \frac{16}{c}\cdot\frac{B(\delta(x_0))}{\delta(x_0)}\cdot C_{u,\lambda}^U.
\end{align*}
The last inequality follows since $\delta(y) < \delta_0$. So,
\[\abs{\nabla(M_0 v)(\widehat{\gamma}_p(t))} \,\le\, \frac{256n}{c^2\delta(x_0)}\cdot\frac{B(\delta(x_0))}{\delta(x_0)}\cdot C_{u, \lambda}^U.\]
A similar calculation yields,
\[\max\left\{\abs{v_{xx}}, \abs{v_{xy}}, \abs{v_{yy}}\right\}(\widehat{\gamma}_p(t)) \le \frac{256n}{c^2\delta(x_0)}\cdot\frac{B(\delta(x_0))}{\delta(x_0)}\cdot C_{u,\lambda}^U.\]
Let us now estimate the only remaining term (the middle term) in \eqref{eq:DaggerDagger};
\begin{align*}
\abs{\nabla v(\widehat{\gamma}_p(t))} &= \abs{\nabla v(\widehat{\gamma}_p(t))} \cdot \frac{\delta(\widehat{\gamma}_p(t))}{B(\delta(\widehat{\gamma}_p(t)))} \cdot \frac{B(\delta(\widehat{\gamma}_p(t)))}{\delta(\widehat{\gamma}_p(t))}\\
&\le \frac{8}{c}\cdot\frac{B(\delta(x_0))}{\delta(x_0)}\cdot C_{u,\lambda}^U.
\end{align*}
Hence, from \eqref{eq:DaggerDagger}, we have
\begin{align*}
\abs{\frac{d^2}{dt^2}\left(v\left(\widehat{\gamma}_p(t)\right)\right)} &\le 4\abs{\nabla(M_0 v)(\widehat{\gamma}_p(t))} + C_1\abs{\nabla v (\widehat{\gamma}_p(t))} \\
 &\hspace{1in} + 2\epsilon\max\left\{\abs{v_{xx}}, \abs{v_{xy}}, \abs{v_{yy}}\right\}(\widehat{\gamma}_p(t))\\
 &\le \frac{B(\delta(x_0))}{\delta(x_0)}\left(\frac{1024n}{c^2\delta(x_0)}(1+\epsilon)+\frac{8}{c}\cdot C_1\right)C_{u,\lambda}^U.
\end{align*}
Using this in \eqref{eq:Dagger} we get,
\[\abs{Mv(x_0)} \le \frac{B(s-\widehat{T}_{x_0})}{s-\widehat{T}_{x_0}} + \frac{s-\widehat{T}_{x_0}}{2}\cdot\frac{B(\delta(x_0))}{\delta(x_0)}\left(\frac{1024n}{c^2\delta(x_0)}(1+\epsilon)+\frac{8}{c}\cdot C_1\right)C_{u,\lambda}^U.\]
Using the choice of $s$ from \eqref{eq:sChoice}, we have a positive constant $C=C(\kappa) = \operatorname{O}(1/\kappa)$,
\[\abs{Mv(x_0)} \le C\cdot\frac{B(\delta(x_0))}{\delta(x_0)} + \kappa\frac{B(\delta(x_0))}{\delta(x_0)}\cdot\left(\frac{1+\epsilon}{8}+\frac{\delta(x_0)}{4}\right)C_{u,\lambda}^U.\]
Later, we will choose $\kappa$ sufficiently small to make the coefficient in front of $C_{u,\lambda}^U$ small. This choice will make $C$ large, but for our purposes it does not matter. Since $\delta(x_0) < \delta_0 < \epsilon$,
\begin{equation}\label{eq:MEst}
\abs{Mv(x_0)}\frac{\delta(x_0)}{B(\delta(x_0))} \le C + \kappa\left(\frac{1+3\epsilon}{8}\right) C_{u,\lambda}^U.
\end{equation}
Note that the above estimate holds for any $x_0\in U'\cap\mathbb{B}$. In what follows, let us restrict $x_0$ further close to $b\mathbb{B}$, i.e., $x_0\in U''\cap\mathbb{B}$, where 
\[U'' := \left\{ x\in\mathbb{R}^n \, :\, \delta(x) < \frac{\delta_0}{10}\right\}.\]

Let us now estimate the derivatives of $v$ in directions orthogonal to $\vec{M_0}$. Let $\vec{N_0}$ be a unit vector orthogonal to $\vec{M_0}$. Let $N_0f = \nabla f \cdot \vec{N_0}$. To estimate $N_0v(x_0)$, 
we use the fundamental theorem of calculus in the $\vec{M_0}$ direction. Since $M_0$ and $N_0$ are constant coefficient, they commute and also preserve harmonic functions. Then, we proceed to use the 
estimate on $M_0$ by noting that $\abs{M_0N_0 v} = \abs{N_0M_0 v} \le \abs{\nabla M_0 v}$ and applying Lemma \ref{lemma:HarEst} to $\abs{\nabla M_0 v}$.
 
Let $y_0 = x_0 + (\delta_0/10)\vec{M_0}$. So,
\begin{align*}
N_0v(x_0) &= N_0v(y_0) - \int\limits_0^{\delta_0/10}\, M_0N_0v(x_0+s\vec{M_0})\, ds, \text{ and hence}\\
\abs{N_0v(x_0)} &\le \abs{N_0v(y_0)} + \int\limits_0^{\delta_0/10}\, \abs{\nabla\left(M_0v\right)(x_0+s\vec{M_0})}\, ds.
\end{align*}
Since $\exists\, P^{M_0}_{x_0} \in b\mathbb{B}$ such that $x_0-P^{M_0}_{x_0}$ is parallel to $M_0$, by Lemma 
\ref{lemma:TransDistEst}, we have
\[\frac{c}{4}\left(\delta(x_0)+s\right)\le\delta(x_0+s\vec{M_0}) \le \delta(x_0) + s.\]
Hence,
\begin{multline}\label{eq:N0Est}
\abs{N_0v(x_0)} \le \abs{N_0v(y_0)}\\
+ \frac{8n}{c}\int\limits_0^{\delta_0}\, \frac{1}{\left(\delta(x_0)+s\right)} \sup\left\{ \abs{M_0v(\xi)}\, :\, \abs{\xi-(x_0+s\vec{M_0})} = \frac{c}{8}\left(\delta(x_0)+s\right)\right\}\, ds.
\end{multline}
For $\abs{\xi-(x_0+s\vec{M_0})} = \dfrac{c}{8}\left(\delta(x_0)+s\right)$, we have, by Lemma \ref{lemma:TransDistEst},
\[\frac{c}{8}\left(\delta(x_0)+s\right) \le \delta(\xi) \le \left(1+\frac{c}{8}\right)\left(\delta(x_0)+s\right), \]
and hence
\[\frac{B(\delta(\xi))}{\delta(\xi)} \le \frac{B\left(\frac{c}{8}\left(\delta(x_0)+s\right)\right)}{\frac{c}{8}\left(\delta(x_0)+s\right)}.\]
Since $\abs{\xi-x_0} < \delta_0/8$, $\abs{M_0v(\xi)} \le \abs{Mv(\xi)} + \epsilon\abs{\nabla v(\xi)}$. Also, since $\delta(\xi) < \delta_0/4$, we can use \eqref{eq:MEst} to estimate $\abs{Mv(\xi)}$ to obtain
\begin{align*}
\abs{M_0v(\xi)} &\le \left(\abs{Mv(\xi)}+\epsilon\abs{\nabla v(\xi)}\right)\frac{\delta(\xi)}{B(\delta(\xi))}\cdot\frac{B\left(\frac{c}{8}\left(\delta(x_0)+s\right)\right)}{\frac{c}{8}\left(\delta(x_0)+s\right)}\\
&\le \frac{B\left(\frac{c}{8}\left(\delta(x_0)+s\right)\right)}{\frac{c}{8}\left(\delta(x_0)+s\right)}\left(C + \frac{\kappa(1+3\epsilon)}{8}C_{u,\lambda}^U+\epsilon C_{u,\lambda}^U\right).
\end{align*}
Since $B$ is non-decreasing and regular,
\[\frac{64n}{c^2}\int\limits_0^{\delta_0}\, \frac{B\left(\frac{c}{8}\left(\delta(x_0)+s\right)\right)}{\left(\delta(x_0)+s\right)^2}\, ds\, \le\, \frac{64n}{c^2}\int\limits_0^{\delta_0}\, \frac{B\left(\delta(x_0)+s\right)}{\left(\delta(x_0)+s\right)^2}\, ds\, \le\, \frac{64nC_2}{c^2}\cdot\frac{B\left(\delta(x_0)\right)}{\delta(x_0)},\]
where $C_2$ is the constant from the inequality \eqref{eq:RegMaj}. Using this in \eqref{eq:N0Est}, we get
\[\abs{N_0v(x_0)} \le \abs{N_0v(y_0)} + \frac{64nC_2}{c^2}\left(C + \frac{\kappa(1+3\epsilon)+8\epsilon}{8}\cdot C_{u,\lambda}^U\right)\frac{B\left(\delta(x_0)\right)}{\delta(x_0)}.\]
Let 
\[C_3 := \sup\left\{ \abs{\nabla u (x)} \, :\, \delta(x) \ge \frac{c\delta_0}{40}\right\}.\]
Since 
\[\delta(y_0) \ge \frac{c}{4}\left(\delta(x_0)+\frac{\delta_0}{10}\right) \ge \frac{c\delta_0}{40},\]
we have
\[\abs{N_0v(x_0)}\frac{\delta(x_0)}{B\left(\delta(x_0)\right)} \le C_3\frac{\delta(x_0)}{B\left(\delta(x_0)\right)}  + \frac{64nC_2}{c^2}\left(C + \frac{\kappa(1+3\epsilon)+8\epsilon}{8}\cdot C_{u,\lambda}^U\right).\]
Now, since $\delta(x_0) < (\delta_0/10) < \delta_0$, and $t/B(t)$ is non-decreasing,
\[\abs{N_0v(x_0)}\frac{\delta(x_0)}{B\left(\delta(x_0)\right)} \le C_3\frac{\delta_0}{B\left(\delta_0\right)}  + \frac{64nC_2}{c^2}\left(C + \frac{\kappa(1+3\epsilon)+8\epsilon}{8}\cdot C_{u,\lambda}^U\right).\]
Combining \eqref{eq:MEst} with the above estimate applied to $(n-1)$ orthonormal directions in $\vec{M_0}^\perp$, we get
\begin{multline*}
\abs{\nabla v(x_0)}\frac{\delta(x_0)}{B\left(\delta(x_0)\right)} \le C_4+ \left\{ \frac{\kappa(1+3\epsilon)}{8} + \frac{64n(n-1)C_2}{c^2}\left(\frac{\kappa(1+3\epsilon)+8\epsilon}{8}\right)\right\} C_{u,\lambda}^U\\
= C_4 + \left\{\left(1+\frac{64n(n-1)C_2}{c^2}\right)\frac{\kappa(1+3\epsilon)}{8} + \frac{64n(n-1)C_2}{c^2}\cdot \epsilon \right\}C_{u,\lambda}^U.
\end{multline*}
for some $C_4>0$. Choose 
\[\kappa = \left(10+\frac{640n(n-1)C_2}{c^2}\right)^{-1}.\]
So, for $x_0\in U''\cap\mathbb{B}$, the choices of $\epsilon$ and $\kappa$ give us
\[\abs{\nabla v(x_0)}\frac{\delta(x_0)}{B\left(\delta(x_0)\right)} \le C_4 + \frac{3}{100} C_{u,\lambda}^U.\]
Taking supremum over $U''\cap\mathbb{B}$,
\[C_{u,\lambda}^{U''} \, \le\,  C_4 + \frac{3}{100} C_{u,\lambda}^U.\]
By Lemma \ref{lemma:EstOutsideCpt},
\[C_{u,\lambda}^{U} \,\le\, \frac{\delta_0}{(\delta_0/10)}\cdot C_{u,\lambda}^{U''} \,\le\, 10C_4 + \frac{3}{10}C_{u,\lambda}^U,\]
and hence
\[C_{u,\lambda}^{U} \,\lesssim\, C_4.\]
Since the constants involved in the inequalities are independent of $\lambda$, let $\lambda \to 1$, to get $C_u^U < \infty$. Now by the Hardy-Littlewood Theorem, $u\in\Lambda_B(\mathbb{B})$.
\end{proof}

As alluded to earlier, notice that all the analysis so far was local, centred around a point near $b\mathbb{B}$. None of this depends on the behaviour of $u$ elsewhere in $\mathbb{B}$ or on the fact that the radius of this ball was $1$. So, we have the following corollary to Theorem \ref{thm:TransLipBall}.

\begin{cor}\label{CorMainThm} Let $\mathbb{B}_R = \{\abs{x} < R\}$ for some $R>0$. Let $\Gamma$ be a family of transversal curves to $b\mathbb{B}_R$. Let $S$ be an 
open $\mathbb{R}^n$-sector in $\mathbb{B}_R$ and $u$ be harmonic in $\mathbb{B}_R$. If $u$ is Lipschitz-$B$ along $\Gamma$ near $bB_R$ in $S$, then there exists a $\mathbb{R}^n$-sub-sector 
$\widetilde{S}$ of $S$ in $\mathbb{B}_R$ such that $u\in\Lambda_{B}(\widetilde{S}\cap U)$, where $U$ is a neighbourhood of b$\mathbb{B}_R$.
\end{cor}
\begin{proof} Let $U$ be a neighbourhood of $b\mathbb{B}_R$ such that $\Gamma$ defines a $C^1$ bijection onto $U$. Restrict $U$, if necessary, so that it satisfies the requirements of the neighbourhood in 
the proof of the above theorem and also so that there exists a a $\mathbb{R}^n$-sub-sector of $S$, call it $\widetilde{S}$, such that
\[\widetilde{S} \cap U \subset \left\{x\in S \cap U\, :\, B(x,\delta_{b\mathbb{B}_R}(x)) \subset S\right\}\qedhere\]
\end{proof}

We use this to prove the Main Theorem.

\begin{mainthm} Let $\Omega \subset\subset \mathbb{R}^n$ have smooth boundary and let $\Gamma$ be a family of curves transversal to $b\Omega$. Let $u\in\Har(\Omega)$ and 
$B$ be a regular majorant. If $u$ is transversally Lipschitz-$B$ with respect to $\Gamma$, then $u\in\Lambda_B(\Omega)$.
\end{mainthm}
\begin{proof}
Let $c>0$ be the transversality constant of $\Gamma$. There exists a neighbourhood $U$ of $b\Omega$ such that the restriction of curves of $\Gamma$ to $U$ defines a $C^1$ bijection onto $U$ and Lemma 
\ref{lemma:TransDistEst} holds in $U\cap\Omega$. $\exists\, s_0>0$ such that, for $p\in b\Omega$, $p-s_0\nu_p \in U\cap\Omega$  and $\displaystyle b\mathbb{B}_p \cap b\Omega = \{p\}$, where 
$\mathbb{B}_p=\mathbb{B}(p-s_0\nu_p,s_0)$.
\begin{figure}[h]
\begin{center}
\includegraphics[width=0.43\textwidth]{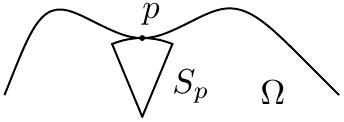}
\caption{Sector $S_p$}
\end{center}
\end{figure}
Fix a $p\in b\Omega$. Let $S_p$ be a $\mathbb{R}^n$-sector of $\mathbb{B}_p$ such that $\Gamma$ is a family of curves transversal to the $\mathbb{R}^n$-spherical boundary of $S_p$ with a transversality constant of at least 
$c/2$. Since $u$ is harmonic in $S_p$ and Lipschitz-$B$ along $\Gamma$, by Corollary \ref{CorMainThm} we have a sub-sector $\widetilde{S}_p$ of $S_p$ and a neighbourhood $V$ of $b\Omega$ such that 
$u\in \Lambda_{B}(\widetilde{S}_p\cap V)$. It is evident that $x\in V\cap\Omega$ belongs to $\widetilde{S}_p$ for some $p\in b\Omega$. This shows that $u\in\Lambda_{B}(\Omega)$.
\end{proof}

\section*{Acknowledgements}

This work is part of my Ph.D. dissertation \cite{Rav11} at The Ohio State University. I am deeply indebted to Jeffery McNeal, my thesis advisor, for his inspiration, motivation, and guidance over the years. I 
would like to thank Kenneth Koenig for his insightful feedback on this work. The exposition here, especially the organization of the proof of Theorem \ref{thm:TransLipBall}, has significantly benefited from his 
input.

\bibliographystyle{plain}

\begin{thebibliography}{1}

\bibitem{Det81}
Jacqueline D{\'e}traz.
\newblock Classes de {B}ergman de fonctions harmoniques.
\newblock {\em Bull. Soc. Math. France}, 109(2):259--268, 1981.

\bibitem{Dya97}
Konstantin~M. Dyakonov.
\newblock Equivalent norms on {L}ipschitz-type spaces of holomorphic functions.
\newblock {\em Acta Math.}, 178(2):143--167, 1997.

\bibitem{Gil-Tru83}
David Gilbarg and Neil~S. Trudinger.
\newblock {\em Elliptic partial differential equations of second order}, volume
  224 of {\em Grundlehren der Mathematischen Wissenschaften [Fundamental
  Principles of Mathematical Sciences]}.
\newblock Springer-Verlag, Berlin, second edition, 1983.

\bibitem{Hav71}
V.~P. Havin.
\newblock A generalization of the {P}rivalov-{Z}ygmund theorem on the modulus
  of continuity of the conjugate function.
\newblock {\em Izv. Akad. Nauk Armjan. SSR Ser. Mat.}, 6(2-3):252--258; ibid. 6
  (1971), no. 4, 265--287, 1971.

\bibitem{Her-McN10}
A.-K. Herbig and J.~McNeal.
\newblock Convex defining functions for convex domains.
\newblock J. Geom. Anal., to appear, arXiv:0912.4653v2.

\bibitem{Pav99}
Miroslav Pavlovi{\'c}.
\newblock On {K}. {M}. {D}yakonov's paper: ``{E}quivalent norms on
  {L}ipschitz-type spaces of holomorphic functions'' [{A}cta {M}ath.\ {\bf 178}
  (1997), no.\ 2, 143--167; {MR}1459259 (98g:46029)].
\newblock {\em Acta Math.}, 183(1):141--143, 1999.

\bibitem{Pav07-LipHarmonic}
Miroslav Pavlovi{\'c}.
\newblock Lipschitz conditions on the modulus of a harmonic function.
\newblock {\em Rev. Mat. Iberoam.}, 23(3):831--845, 2007.

\bibitem{Rav11}
Sivaguru Ravisankar.
\newblock {\em Lipschitz properties of harmonic and holomorphic functions}.
\newblock Ph.{D}. diss., The Ohio State University, 2011.

\bibitem{Zyg59}
A.~Zygmund.
\newblock {\em Trigonometric series. 2nd ed. {V}ols. {I}, {II}}.
\newblock Cambridge University Press, New York, 1959.

\end{thebibliography}

\listoffigures

\end{document}